\documentclass[reqno]{amsart}

\usepackage{amssymb}

\theoremstyle{plain}
\newtheorem{theorem}{Theorem}[section]

\newtheorem{lemma}[theorem]{Lemma}

\theoremstyle{definition}
\newtheorem{definition}[theorem]{Definition}

\theoremstyle{remark}

\numberwithin{equation}{section}

\begin{document}

\title[Three term relations for a class of bivariate polynomials]{Three term relations for 
a class of \\ bivariate orthogonal polynomials}

\author[M. Marriaga, T. E. P\'erez and M. A.  Pi\~nar]
{Misael Marriaga, Teresa E. P\'erez{\textsuperscript{*}} and Miguel A.  Pi\~nar}
\address[M. Marriaga]{Department of Mathematics.
University Carlos III de Madrid. 28911. Legan\'es (Madrid), Spain}
\email{mmarriag@math.uc3m.es}
\address[T. E. P\'erez, M. A. Pi\~nar]{Math Institute of the University of Granada -- IEMath-GR \& 
Department of Applied Mathematics.  
University of Granada. 18071. Granada, Spain}
\email{tperez@ugr.es, mpinar@ugr.es}

\begin{abstract}
We study matrix three term relations for orthogonal polynomials in two variables 
constructed from orthogonal polynomials in one variable. Using the three term 
recurrence relation 
for the involved univariate orthogonal polynomials, the
explicit expression for the matrix coefficients in these three term 
relations are deduced. These matrices are diagonal or tridiagonal with entries
computable from the one variable coefficients in the respective three 
term recurrence relation. Moreover, some interesting particular cases are considered.
\end{abstract}



\thanks{This work has been partially supported by MINECO of Spain and the European 
Regional Development Fund (ERDF) through grant 
MTM2014--53171--P, and Jun\-ta de Andaluc\'{\i}a grant P11--FQM--7276 and Research Group FQM--384.}

\thanks{*\,\, Corresponding author. E--mail: \texttt{tperez@ugr.es}}
\date{\today}

\maketitle

\section{Introduction}

In 1975, T. Koornwinder (\cite{Koor75}) studied examples of two variables analogues of the Jacobi 
polynomials, and introduced seven classes of orthogonal polynomials in two variables obtained from 
Jacobi weights. All seven classes of orthogonal polynomials are eigenfunctions of two commuting and
algebraically independent partial differential operators $D_1$ and $D_2$, where $D_1$ has order two, 
and $D_2$ may have any arbitrary order, Koornwinder considered those orthogonal polynomials as 
{\it two--variable analogues of the Jacobi polynomials}. The examples studied by 
Koornwinder included four families of bivariate orthogonal polynomials that could be expressed in
terms of Jacobi polynomials, namely, the tensor product of Jacobi polynomials in one variable, 
classical orthogonal polynomials on the unit disk, classical orthogonal polynomials on the simplex, 
and a class of orthogonal polynomials on the parabolic biangle. Using an interesting tool previously
introduced by C. A. Agahanov (\cite{Ag65}), Koornwinder constructed bases of orthogonal polynomials 
in two variables from univariate orthogonal polynomials. This tool can be used to construct 
orthogonal polynomials bases associated with a particular class of weight functions defined on 
either bounded or unbounded domains (see \cite{FPP12}).

Our first goal in this paper was to extend the Agahanov construction to orthogonal polynomials
defined from quasi--definite moment functionals. This objective was achieved using a similar
construction to that given by Kwon, Lee, and Littlejohn in \cite{KLL01}. In this paper the authors 
studied the classical bivariate orthogonal polynomials considered by Krall and Sheffer  in \cite{KS67}.
In fact, Kwon, Lee, and Littlejohn proved that eight of the nine classes in \cite{KS67} fit in the
Agahanov construction, some of them are orthogonal with respect to a non--positive definite
moment functional.

As it is well known, orthogonal polynomials in two variables satisfy a three term relation in 
each variable (\cite{DX14}). If these three term relations are written in vector form, they consist on two
three term relations with matrix coefficients. Naturally, the structure of the matrices depend on
the particular choice of the bivariate orthogonal polynomials. In this 
paper, we study the three term relations of a general family of orthogonal polynomials constructed 
by means of an extension of the Agahanov method, and determine 
the structure of the matrix coefficients. Using the three term recurrence relation for the involved 
univariate orthogonal polynomials, we deduce the explicit expressions of the matrix entries for the examples 
of two variable analogues of the Jacobi polynomials studied by Koornwinder. Also, we provide
the matrix coefficients for the three term recurrence relations in some of the non--positive
definite classical orthogonal polynomials in \cite{KS67}. The coefficients in the ball, simplex
ans squared cases were obtained by Y. Xu in \cite{X94} by using Jacobi orthonormal polynomials.

The structure of this paper is as follows. In section 2, preliminaries and definitions on 
orthogonal polynomial systems 
in two variables relevant to this paper will be given. In section 3, the method introduced by Agahanov 
to generate orthogonal polynomials systems in two variables from orthogonal polynomials 
sequences in one variable will be discussed and extended to quasi--definite moment functionals. 
In section 4 we comment the three term relations for 
polynomials in two variables and we elaborate the details of our main results. Finally, section 
5 contains as illustrative examples four instances of Jacobi polynomials in two variables included in
Koornwinder's work, as well as a non--positive definite example studied in \cite{KLL01}. 
All the necessary formulae for Jacobi, Laguerre and Bessel polynomials have been reunited in one
appendix at the end of the paper.

\bigskip


\section{Basic theory}

First, we will establish the main notations and results that we will use later. Chihara's (\cite{Ch78})
and Sze\H o (\cite{Sz78}) books are the main reference for this part.

Let us denote by $\Pi$ the linear space of univariate polynomials with real coefficients, and let
$u$ be a linear functional defined by means of its moments. For a given sequence of real numbers
$\{\mu_n\}_{n\ge 0}$ we define $\langle u, x^n\rangle = \mu_n$, for $n\ge0$, and extend it
by linearity to all polynomials. We will call say that $u$ a \emph{univariate moment functional}.

A sequence of univariate polynomials $\{p_n\}_{n\ge0}$ 
is an \emph{orthogonal polynomial sequence} associated with a moment functional $u$ if $\deg p_n = n$, 
for $n\ge 0$, and 
$$\langle u, p_n\,p_m\rangle = h_n\,\delta_{n,m}, \quad n,m\ge 0,$$
with $h_n\neq 0, \,\, n\ge 0$.  

The moment functional $u$ is called quasi--definite if there exists an orthogonal polynomial sequence
(OPS) associated with $u$ unique up to a normalizing constant. 

Moreover, 
if $\langle u, p^2\rangle >0$, for all $p(x)\in\Pi$, $p\neq 0$, then $u$ is called \textit{positive 
definite}. Of course, positive--definite implies quasi--definite. 
In this situation, we can obtain a sequence of real \emph{orthonormal} polynomials, that is,
$$
\langle u, p_n\,p_m\rangle = \delta_{n,m}, \quad n,m\ge 0.
$$

A quasi--definite moment functional is called \emph{symmetric} if $\langle u, x^{2n+1}\rangle =0$, 
for $n\ge 0$. If $\{p_n\}_{n\ge0}$ is an \emph{orthogonal polynomial
sequence} associated with a symmetric moment functional $u$, then polynomials are even or odd 
functions according with the parity of their indices. 

Sometimes in this paper, we will be interested in positive--definite moment functionals 
that can be
represented by means of a weight function $w(x)$ over a real interval $[a,b]$ such as
$$
\langle u, p(x)\rangle = \int_a^b \,p(x)\,w(x)\,dx, \quad \forall p(x)\in\Pi.
$$

The {\it left multiplication} of a polynomial $q\in \Pi$ by $u$ is the moment functional 
$q\,u$ satisfying 
$$
\langle q \, u, p \rangle = \langle u, q\, p \rangle,\quad \forall p\in \Pi.
$$

\bigskip

Now, some basic theory of bivariate orthogonal polynomials is introduced for its use 
in the sequel. We follow mainly \cite{DX14}.

For $n\ge 0$, let $\Pi^2_n$ denote the linear space of real polynomials in two variables of total degree not 
greater than $n$, and let $\Pi^2=\bigcup_{n\ge 0} \Pi^2_n$ the collection of all bivariate real polynomials.

A useful tool in the theory of orthogonal polynomials in several variables is the representation of a 
basis of polynomials as a {\it polynomial system} (PS).

\begin{definition}
A {\emph polynomial system (PS)} is a sequence of column vectors of increasing size $n+1$, 
$\{\mathbb{P}_n\}_{n \ge 0}$, whose entries are independent polynomials of total 
degree $n$ 
\begin{equation}\label{**}
\mathbb{P}_n = \mathbb{P}_n(x,y) = \Big( P_{n,0}(x,y), P_{n,1}(x,y), \ldots,P_{n,n}(x,y) \Big)^t.
\end{equation}
\end{definition}

Observe that, for $n\ge 0$, the entries in $\{\mathbb{P}_0, \mathbb{P}_1, \dots, \mathbb{P}_n \}$ form a 
basis of $\Pi_n^2$.

\bigskip

Let us return to moment functionals and orthogonality in two variables. For a given sequence of 
real numbers $\left\{\omega_{h,k}\right\}_{h,k\ge0}$, a moment 
functional $w$ is defined by means of its moments
$$\langle w, x^h\,y^k \rangle = \omega_{h,k},
$$
and extended by linearity to all bivariate polynomials.

\bigskip

Throughout this paper, we will denote by $\mathcal{M}_{h \times k}(\mathbb{R})$, respectively 
$\mathcal{M}_{h \times k}(\Pi^2)$, the linear space of matrices 
of size $h\times k$ with real entries, resp. with polynomial entries, 
and the notation will be simplified when $h=k$ as 
$\mathcal{M}_{h}$.
Given a matrix $M\in \mathcal{M}_{h \times k}$ 
we will denote its transpose by $M^t$. As usual, we will say
that $M\in\mathcal{M}_h$ is non--singular (or invertible) if $\det M\neq 0$, and symmetric if 
$M^t = M$. 
Moreover, $I_h$ will denote the identity matrix of size $h$, and we will omit the subscript when the size 
is clear from the context.

The {\it action of $u$ on a polynomial matrix} is defined by
$$
\langle u, M \rangle =\left(\langle u, m_{i,j}(x)\rangle
\right)_{i,j=1}^{h,k}\in {\mathcal M}_{h\times k}(\mathbb{R}),$$
where $M=\left(m_{i,j}(x)\right)_{i,j=1}^{h,k} \in {\mathcal M}_{h\times
k}(\Pi^2)$.

\bigskip

Some basic operations acting on a bivariate moment functional are described now. The
{\it action of $w$ on a polynomial matrix} is defined by
$$
\langle w, M \rangle =\left(\langle w, m_{i,j}(x,y)\rangle
\right)_{i,j=1}^{h,k}\in {\mathcal M}_{h\times k}(\mathbb{R}),$$
where $M=\left(m_{i,j}(x,y)\right)_{i,j=1}^{h,k} \in {\mathcal M}_{h\times
k}(\Pi^2),$ and the {\it left multiplication} of a polynomial $q\in \Pi^2$ by $w$ 
is defined as follows 
$$\langle q \, w, p \rangle = \langle w, q\, p \rangle,\quad \forall p\in \Pi^2.
$$

\bigskip

We will say that a polynomial system $\{\mathbb{P}_n\}_{n \ge 0}$ is \emph{orthogonal (OPS)} with respect
to $w$ if 
$$
\langle w, \mathbb{P}_n\, \mathbb{P}_m^t\rangle = \left\{ \begin{array}{ll}
\mathtt{0} \in\mathcal{M}_{(n+1)\times(m+1)}(\mathbb{R}) & \quad m\ne n,\\
H_n \in\mathcal{M}_{n+1}(\mathbb{R})& \quad m=n, \end{array} \right.$$
where $H_n$ is a symmetric and non--singular matrix of size $n+1$. 

When $H_n$ is a diagonal matrix, we say that the OPS is a \emph{mutually orthogonal 
polynomial system}. 

A bivariate moment functional $w$ is quasi--definite if there exists an orthogonal polynomial system
associated with $w$. We must remark that in the bivariate case, OPS are not unique. 

As in the univariate
case, $w$ is called \emph{positive definite} if $\langle w, p^2\rangle > 0$ for all $p\in\Pi^2$, $p\neq 0$.
If $w$ is positive definite, then it is quasi--definite, and we can construct \emph{orthonormal
polynomial systems} satisfying $H_n=I_{n+1}$, $\forall n\ge 0$.

\bigskip

Orthogonal polynomials in two variables satisfy a three term relation in each variable
(\cite{DX14}). These three term relations are written in vector form and have matrix coefficients.

\begin{theorem}[\cite{DX14}, p. 70]
Let $\{\mathbb{P}_n\}_{n\ge 0}$ be an OPS associated with the bivariate moment functional $w$. 
For $n\ge 0$, there exist matrices $A_{n,i}$ of size $(n+1)\times (n+2)$, $B_{n,i}$ of 
size $(n+1)\times (n+1)$, and $C_{n,i}$ of size $(n+1)\times n$, for $i=1,2$, such that
\begin{eqnarray}
x \, \mathbb{P}_n &= A_{n,1}\,\mathbb{P}_{n+1} + B_{n,1}\,\mathbb{P}_n + C_{n,1}\,\mathbb{P}_{n-1}, 
\label{r3t-i1}\\
y \, \mathbb{P}_n &= A_{n,2}\,\mathbb{P}_{n+1} + B_{n,2}\,\mathbb{P}_n + C_{n,2}\,\mathbb{P}_{n-1},
\label{r3t-i2}
\end{eqnarray}
where $\mathbb{P}_{-1}=0$ and $C_{-1,i}=0$, $i=1,2$, and 
$$
\begin{array} {lllllll}
A_{n,1}\,H_{n+1} & = & \langle w, x\,\mathbb{P}_n\,\mathbb{P}_{n+1}^t\rangle, & \quad & 
A_{n,2}\,H_{n+1} & = & \langle w, y\,\mathbb{P}_n\,\mathbb{P}_{n+1}^t\rangle, \\
B_{n,1}\,H_n & = & \langle w, x\,\mathbb{P}_n\,\mathbb{P}_{n}^t\rangle, & & 
B_{n,2}\,H_n & = & \langle w, y\,\mathbb{P}_n\,\mathbb{P}_{n}^t\rangle, \\
C_{n,1}\,H_{n-1} & = & H_{n}\,A^t_{n-1,1}, & & C_{n,2}\,H_{n-1} & = & H_{n}\,A^t_{n-1,2}.
\end{array} 
$$
Moreover, for $n\ge 0$ and $i=1, 2,$ the matrices $A_{n,i}$ and $C_{n+1,i}$ satisfy the rank conditions
\begin{equation}\label{rango1}
\mathrm{rank} \, A_{n,i} = \mathrm{rank} \, C_{n+1,i}=n+1,
\end{equation}
and 
\begin{equation}\label{rango2}
\mathrm{rank} \, A_{n} = \mathrm{rank} \, C_{n+1}^t = n+2,
\end{equation}
where 
$$A_n=\left(\begin{array}{c}
  A_{n,1} \\
  A_{n,2}
\end{array}\right) \in \mathcal{M}_{2(n+1)\times (n+2)}(\mathbb{R}), \hspace{0.5 cm} C_{n} = 
\left(
  C_{n,1},  C_{n,2}\right) \in \mathcal{M}_{(n+1)\times 2\,n}(\mathbb{R})$$
are called \emph{the joint matrices of} $A_{n,i}$ and $C_{n,i}$, respectively.
\end{theorem}

\bigskip

If $w$ is positive--definite and the 
OPS $\{\mathbb{P}_n\}_{n\ge0}$ is orthonormal, then $H_n =I_{n+1}$, for $n\ge 0$, and 
$$
C_{n+1,i} = A_{n,i}^t, \quad i=1,2.
$$
Therefore, the three term relations \eqref{r3t-i1}--\eqref{r3t-i2} are written as
\begin{equation}\label{rec-orto}
\begin{aligned}
x\,\mathbb{P}_n &= A_{n,1}\,\mathbb{P}_{n+1} + B_{n,1}\,\mathbb{P}_n + A_{n-1,1}^t\,\mathbb{P}_{n-1}, \\
y\,\mathbb{P}_n &= A_{n,2}\,\mathbb{P}_{n+1} + B_{n,2}\,\mathbb{P}_n + A_{n-1,2}^t\,\mathbb{P}_{n-1}.
\end{aligned}
\end{equation}

\bigskip

We say that a bivariate moment functional $w$ is \emph{centrally symmetric} (\cite[p. 76]{DX14}) if
all the moments of odd order vanish
$$\langle w, x^h\,y^k\rangle = 0, \quad h, k\ge 0, \quad h+k = \textrm{odd~~ integer}.
$$
As in the univariate case, the properties of symmetry from the inner product can be related with the 
coefficient matrices of the three term relations. In \cite[p. 77]{DX14}, it is shown that a 
moment functional $w$ is centrally symmetric if and only if the matrices $B_{n,i}\equiv 0$ \, for all 
$n\ge 0$ and $i=1,2$.

\bigskip

\section{A class of two variable orthogonal polynomials}\label{KoorConstruction}

In this section, we describe a method (\cite{AS72} and \cite{Koor75}) to construct non trivial 
bivariate orthogonal polynomials from univariate orthogonal polynomials. This method is carefully 
described in \cite[p. 38]{DX14} for weight functions. We will extend it to quasi--definite moment 
functionals using similar ideas as the used ones in \cite{KLL01}. For 
positive definite moment functionals described in terms of weight functions, we recover the original 
Agahanov's method.

\bigskip

Let $u_x$ and $v_y$ be univariate quasi--definite moment functionals acting on the variables $x$ 
and $y$, respectively. Let $\rho(x)$ be a univariate function satisfying one of the following two conditions

\bigskip

\leftline{
\begin{tabular}{lp{10cm}}
\emph{Case I}: & $\rho(x)$ is a polynomial of degree $\le 1$, that is, $\rho(x) = r_1\,x+r_0$, 
with $|r_1|+|r_0|>0$,\\
\emph{Case II}: & $\rho(x)$ is the square root of a polynomial of degree at most 2, 
and $v_y$ is a symmetric moment functional,
\end{tabular}}

\noindent
and such that $u^{(m)}_x = \rho^{2m+1}(x)\,u_x$ are quasi--definite moment functionals, for $m\ge 0$, in both 
Case I and Case II.

\bigskip

Anyway, $\rho(x)^2$ is a polynomial of degree less than or equal to 2, and from now on, 
we will denote
\begin{equation}\label{rho}
\rho(x)^2 = s_2\, x^2 + s_1\, x + s_0,
\end{equation}
its explicit expression, with $s_2, s_1, s_0 \in \mathbb{R}$, and $|s_2|+|s_1| + |s_0|>0$. Observe that, in 
the \emph{Case I} $s_2 = r_1^2\ge 0$, $s_1 = 2\,r_1\,r_0$, and $s_0 = r_0^2\ge 0$.

\bigskip

Let us denote by $w$ as the bivariate moment functional defined by means of its moments
\begin{equation}\label{w-def}
\langle w, x^h\,y^k\rangle = \langle \rho(x)^{k+1}\,u_x,x^h\rangle\,\langle v_y, y^k\rangle, \quad
h,k\ge 0,
\end{equation}
and extended by linearity. Then, it is easy to check that 
$$\langle w, p(x,y)\rangle = \langle u_x, \;\langle v_y, \rho(x)\,p(x,\rho(x)\,y)\rangle\;\rangle = 
\langle \rho(x)\,u_x, \;\langle v_y, p(x,\rho(x)\,y)\rangle\;\rangle,$$
for all $p(x,y)\in\Pi^2$.

\bigskip

For $m\ge 0$, let $\{p_{n}^{(m)}(x)\}_{n\ge0}$ be an orthogonal polynomial sequence 
with respect to the quasi--definite moment functional $u^{(m)}_x = \rho(x)^{2m+1}\,u_x$,
and let $\{q_n(y)\}_{n\ge0}$ be an orthogonal polynomial sequence with respect 
to the moment functional $v_y$. 

\bigskip

Then, we define the polynomials 
\begin{equation}\label{kops}
P_{n,m}(x,y) = p_{n-m}^{(m)}(x)\,\rho(x)^m\, q_m\left(\frac{y}{\rho(x)}\right), \quad 0 \le m\le n.
\end{equation}

Clearly, in \emph{Case I}, $P_{n,m}(x,y)$ is a bivariate polynomial of degree $n-m$ in the first 
variable $x$ and degree $m$ in $y$. In \emph{Case II} the same result can be deduced from 
the fact that $q_m(y)$ has the same parity as $m$, as a consequence of the symmetry of $v_y$.

\begin{theorem}

The bivariate moment functional $w$ defined in \eqref{w-def} is quasi--definite, and the set of 
polynomials $\{P_{n,m}(x,y): n, m\ge 0\}_{n,m\ge 0}$ defined in \eqref{kops} constitutes  
a mutually orthogonal polynomial system with respect to $w$. Moreover, if $u^{(m)}_x$, for $m\ge 0$,
and $v_y$ are positive--definite, then $w$ is positive--definite.

\end{theorem}

\begin{proof}

We will compute the action of $w$ over the product of two polynomials given by \eqref{kops}:
\begin{eqnarray*}
\langle w, P_{n,m}(x,y)P_{h,k}(x,y) \rangle &=& \langle u_x,\; \langle v_y, \rho(x)\,
P_{n,m}(x,\rho(x)y) \,P_{h,k}(x,\rho(x)y)\rangle \;\rangle\\
&=& \langle u_x, p_{n-m}^{(m)}(x)\,p_{h-k}^{(k)}(x)\,\rho(x)^{m+k+1}\; 
\langle v_y, q_m(y)\, q_k(y)\rangle\;\rangle\\
&=& \langle \rho(x)^{m+k+1}\,u_x,p_{n-m}^{(m)}(x)\,p_{h-k}^{(k)}(x)\rangle\,
\tilde{h}_m\,\delta_{m,k}\\
&=& h^{(m)}_{n-m}\,\tilde{h}_m\,\delta_{n,h}\,\delta_{m,k},
\end{eqnarray*}
where 
$$
h_{n-m}^{(m)} = \langle u^{(m)}_x, p_{n-m}^{(m)}(x)\,p_{n-m}^{(m)}(x)\rangle, \quad 
\tilde{h}_m\, = \langle v_y, q_{m}(y)\,q_{m}(y)\rangle, \quad 0\le m\le n,
$$
both quantities are different from zero because of the quasi--definitive character of the moment functionals.

Therefore, $w$ is quasi--definite since we have found a mutually orthogonal polynomial system. Moreover,
if we define
$$
\mathbf{h}_{n,m} = \langle w, P_{n,m}(x,y) \,P_{n,m}(x,y) \rangle = h^{(m)}_{n-m}\,\tilde{h}_m,
$$
then $w$ is positive definite if $u^{(m)}_x$ and $v_y$ are positive definite.
\end{proof}

\bigskip

From now on, we will consider an orthogonal polynomial system $\{\mathbb{P}_n\}_{n\ge 0}$ 
$$
\mathbb{P}_n=(P_{n,0}(x,y), P_{n,1}(x,y), \dots, P_{n,n}(x,y))^t, \quad n\ge 0,
$$
such that $P_{n,m}(x,y)$, $0\le m \le n$, is constructed using method (\ref{kops}).

\bigskip

\section{Three term relations for bivariate polynomials}

In this section, we will deduce explicit expressions for the matrix coefficients of 
the three term relations for polynomials defined in (\ref{kops}). First, we need 
to stablish some properties for the univariate polynomials.

\bigskip

Suppose that the explicit expressions for the univariate orthogonal polynomial sequences
$\{p_n^{(m)}(x)\}_{n\ge0}$ and $\{q_n(y)\}_{n\ge0}$ are given by
$$
\begin{array}{lll}
p_n^{(m)}(x) & = & k_n^{(m)}\,x^n + l_n^{(m)}\, x^{n-1} + \textrm{lower degree terms}, \quad m\ge 0,\\
q_n(y) & = & \tilde{k}_n\,y^n + \tilde{l}_n\, y^{n-1}+ \textrm{lower degree terms}.
\end{array}
$$
An orthogonal polynomial sequence in one variable satisfy a three term recurrence 
relation (\cite{Ch78}, \cite{Sz78}). We write these relations for our families
\begin{eqnarray}
x\, p_n^{(m)}(x) &=& a_n^{(m)}\, p_{n+1}^{(m)}(x) + b_n^{(m)}\, p_n^{(m)}(x) + c_{n}^{(m)}\,
p_{n-1}^{(m)}(x),\quad n\ge 0,
\label{recurrence1}\\
p_{-1}^{(m)}(x)& =& 0, \quad p_0^{(m)}(x)=1, \quad m\ge 0,\nonumber
\end{eqnarray}
where
\begin{equation}\label{rrcoefp}
a_n^{(m)} =  \frac{k_n^{(m)}}{k_{n+1}^{(m)}}, \qquad b_n^{(m)} = \frac{l_n^{(m)}}{k_n^{(m)}} -
\frac{l_{n+1}^{(m)}}{k_{n+1}^{(m)}}, \qquad c_n^{(m)} = a^{(m)}_{n-1}\,\frac{h^{(m)}_n}{h^{(m)}_{n-1}},
\end{equation}
and
\begin{eqnarray}
y\,q_n(y) & =& \tilde{a}_n \, q_{n+1}(y) + \tilde{b}_n\,q_n(y) + \tilde{c}_{n}\,q_{n-1}(y), \quad n\ge 0,
\label{recurrence2}\\
q_{-1}(y) &=& 0, \quad q_0(y)=1,\nonumber
\end{eqnarray}
with 
\begin{equation}\label{rrcoefq}
\tilde{a}_n  = \frac{\tilde{k}_n}{\tilde{k}_{n+1}}, \qquad \tilde{b}_n  = \frac{\tilde{l}_n}{\tilde{k}_n} -
\frac{\tilde{l}_{n+1}}{\tilde{k}_{n+1}}, \qquad \tilde{c}_n = \tilde{a}_{n-1}\,
\frac{\tilde{h}_n}{\tilde{h}_{n-1}}.
\end{equation}
The moment functional $u_x^{(m)}$ is symmetric, respectively $v_y$ is symmetric, if and only if 
$b_n^{(m)}=0$, respectively $\tilde{b}_n=0$, $n\ge 0$.

\bigskip

In our next result, we relate two families of orthogonal polynomials $\{p_n^{(m)}\}_{n\ge 0}$ for 
consecutive values of the integer $m$. 

\begin{lemma}\label{lemma31}
For $m\ge 0$, let $\{p_{n}^{(m)}\}_{n\ge 0}$ and $\{p_{n}^{(m+1)}\}_{n\ge 0}$ be 
univariate sequences of orthogonal polynomials associated with the quasi--definite moment functionals 
$u_x^{(m)}$ and $u_x^{(m+1)}$. Then
\begin{eqnarray}
p_{n}^{(m)}(x) &=& \delta_{n}^{(m)}\,p_{n}^{(m+1)}(x) + \epsilon_{n}^{(m)}\,p_{n-1}^{(m+1)}(x) +
\zeta_{n}^{(m)}\,p_{n-2}^{(m+1)}(x),\label{eq:recurrel1}\\
\rho(x)^2\, p_{n}^{(m+1)}(x) &=& \eta_{n}^{(m)}\, p_{n+2}^{(m)}(x) +
\theta_{n}^{(m)} \,p_{n+1}^{(m)}(x) + \vartheta_{n}^{(m)}\,p_{n}^{(m)}(x),\label{eq:recurrel2}
\end{eqnarray}
where 
\begin{equation}\label{eq3.1}
\delta_n^{(m)} = \frac{k_n^{(m)}}{k_n^{(m+1)}}, \quad 
\epsilon_{n}^{(m)} = \frac{l_n^{(m)}}{k_{n-1}^{(m+1)}} - \frac{k_n^{(m)}}{k_{n}^{(m+1)}}\,
\frac{l_n^{(m+1)}}{k_{n-1}^{(m+1)}}, \quad \zeta_{n}^{(m)} = s_2\, \frac{k_{n-2}^{(m+1)}}{k_n^{(m)}}
\,\frac{h_{n}^{(m)}}{h_{n-2}^{(m+1)}},
\end{equation}
and
\begin{equation}\label{eq3.2}
\eta_{n}^{(m)} = s_2\, \frac{k_{n}^{(m+1)}}{k_{n+2}^{(m)}}, \qquad 
\theta_{n}^{(m)}= \epsilon_{n+1}^{(m)}\,\frac{h_n^{(m+1)}}{h_{n+1}^{(m)}}, \qquad 
\vartheta_{n}^{(m)}= \delta_n^{(m)}\,\frac{h_n^{(m+1)}}{h_n^{(m)}}.
\end{equation}

\end{lemma}

\bigskip

\begin{proof}

First, we express $p_{n}^{(m)}(x)$ in terms of the polynomials $\{p_{i}^{(m+1)}\}_{i\ge 0}$
as follows
\begin{equation}\label{eq:lcomb1}
p_{n}^{(m)}(x) = \sum_{i=0}^{n} \, d_{i}^{(n)}(m)\,p_{i}^{(m+1)}(x),
\end{equation}
for a fixed non--negative integer $m$, and $n\ge 0$. Then, 
$$
d_{i}^{(n)}(m) \, h_i^{(m+1)} = \langle u_x^{(m+1)}, p_{n}^{(m)}\, p_{i}^{(m+1)} \rangle =
 \langle u_x^{(m)}, p_{n}^{(m)}\, \rho(x)^2\,p_{i}^{(m+1)} \rangle.
$$
Observe that $\rho(x)^2$ is a polynomial of degree at most $2$, and then $d_{i}^{(n)}(m)=0$ 
when $i+2 < n$. Comparing leading coefficients, we obtain
$$
\delta_n^{(m)} = d_{n}^{(n)}(m) = \frac{k_n^{(m)}}{k_n^{(m+1)}}.
$$
For $i=n-2$, we get
\begin{eqnarray*}
d_{n-2}^{(n)}(m)\, h_{n-2}^{(m+1)} &=& \langle u_x^{(m+1)}, p_{n}^{(m)}\, p_{n-2}^{(m+1)} \rangle = 
\langle u_x^{(m)}, p_{n}^{(m)}\, \rho(x)^2\,p_{n-2}^{(m+1)} \rangle
\\
& = & s_2\, \frac{k_{n-2}^{(m+1)}}{k_n^{(m)}}\,h_{n}^{(m)},
\end{eqnarray*}
then we get
$$\zeta_n^{(m)} = d_{n-2}^{(n)}(m) = s_2\, \frac{k_{n-2}^{(m+1)}}{k_n^{(m)}}\,
\frac{h_{n}^{(m)}}{h_{n-2}^{(m+1)}},$$
using the explicit expression for $\rho(x)^2$ \eqref{rho}. When $i=n-1$, identifying coefficients 
in the explicit expressions of the polynomials involved in \eqref{eq:recurrel1}, 
we get
$$\epsilon_n^{(m)}\, k_{n-1}^{(m+1)} = l_n^{(m)} - \delta_n^{(m)}\,l_n^{(m+1)},$$
and the expression for $\epsilon_n^{(m)}$ is deduced.

\bigskip

\noindent
In a similar way, for $m\ge 0$, there exist constants $e_{i}^{(n)}(m)$, $0\le i \le n+2$, such that

\begin{equation}\label{lincom1}
\rho(x)^2\,p_{n}^{(m+1)}(x) = \sum_{i=0}^{n+2}\, e_{i}^{(n)}(m)\, p_{i}^{(m)}(x),
\end{equation}
where
$$
e_{i}^{(n)}(m)\,h_i^{(m)} = \langle u_x^{(m)}, \rho(x)^2\, p_{n}^{(m+1)}\, p_{i}^{(m)} \rangle
= \langle u_x^{(m+1)}, p_{n}^{(m+1)}\, p_{i}^{(m)} \rangle.
$$
Then, $e_{i}^{(n)}(m)=0$ for $i<n$. If $i=n+2$,
\begin{equation*}
e_{n+2}^{(n)}(m)\,h_{n+2}^{(m)} = \langle u_x^{(m)}, \rho(x)^2\,p_{n}^{(m+1)}\,p_{n+2}^{(m)}
\rangle = s_2\,
\frac{k_n^{(m+1)}}{k_{n+2}^{(m)}} \,h_{n+2}^{(m)}, 
\end{equation*}
and therefore
$$
\eta_{n}^{(m)} = e_{n+2}^{(n)}(m) = s_2\,\frac{k_n^{(m+1)}}{k_{n+2}^{(m)}} = \zeta_{n+2}^{(m)}\,
\frac{h_n^{(m+1)}}{h_{n+2}^{(m)}}.
$$
When $i=n+1$, we get
$$e_{n+1}^{(n)}(m)\,h_{n+1}^{(m)}=\langle u_x^{(m+1)}, p_{n}^{(m+1)}\, p_{n+1}^{(m)} \rangle = 
\epsilon_{n+1}^{(m)}\,h_{n}^{(m+1)},
$$
thus
$$\theta_{n}^{(m)} = e_{n+1}^{(n)}(m) = \epsilon_{n+1}^{(m)}\,\frac{h_{n}^{(m+1)}}{h_{n+1}^{(m)}}.
$$
Finally, 
$$
e_{n}^{(n)}(m)\,h_n^{(m)} =\langle u_x^{(m+1)},p_{n}^{(m+1)}, p_{n}^{(m)} \rangle = 
\delta_n^{(m)} \,h_n^{(m+1)},
$$
that is, 
$$\vartheta_n^{(m)} = e_{n}^{(n)}(m) = \delta_n^{(m)} \,\frac{h_n^{(m+1)}}{h_n^{(m)}}.$$
\end{proof}

\bigskip

Three term relations for the bivariate polynomials described in (\ref{kops}) will be stated in the 
following two theorems. The coefficients in the \emph{first three term relation} corresponding to the 
$x$ first variable are diagonal matrices.

\begin{theorem}\label{xr3t}
Let $\{\mathbb{P}_n\}_{n\ge0}$ be an orthogonal PS constructed by means of \eqref{kops}. 
The matrix coefficients in the first three term relation \eqref{r3t-i1} are given by
\begin{equation*}
A_{n,1} = \left(
\begin{array}{cccc|c}
a_{n}^{(0)} &  & & \bigcirc & 0\\
 & a_{n-1}^{(1)} & & & \vdots \\
  & & \ddots &  & \vdots\\
\bigcirc  & & & a_{0}^{(n)} & 0\\
\end{array}\right), \, B_{n,1} = \left(
\begin{array}{cccc}
b_{n}^{(0)} & & & \bigcirc \\
 & b_{n-1}^{(1)}  & & \\
 & &  \ddots & \\
\bigcirc & & & b_{0}^{(n)}
\end{array}
\right),
\end{equation*}

$$
C_{n,1} = \left(
\begin{array}{cccc}
c_{n}^{(0)} &  & & \bigcirc\\
 & c_{n-1}^{(1)} & &  \\
  & & \ddots &  \\
\bigcirc  & & & c_{1}^{(n-1)} \\
\hline
0 & \ldots & \ldots & 0
\end{array}\right),
$$
where $a_{n-m}^{(m)}$, $b_{n-m}^{(m)}$, and $c_{n-m}^{(m)}$, for $0\le m\le n$, are defined in \eqref{rrcoefp}.

\end{theorem}

\bigskip

\begin{proof}
If we multiply (\ref{kops}) times $x$, the three term recurrence relation (\ref{recurrence1}) gives
\begin{eqnarray*}
\lefteqn{x\,P_{n,m}(x,y) = x\,p_{n-m}^{(m)}(x)\,\rho(x)^m\,q_m\left(\frac{y}{\rho(x)}\right)}\\
             &=& [a_{n-m}^{(m)}p_{n-m+1}^{(m)}(x)+b_{n-m}^{(m)}p_{n-m}^{(m)}(x) + c_{n-m}^{(m)}
             p_{n-m-1}^{(m)}(x)]\\
             &~& \times \rho(x)^m\,q_m\left(\frac{y}{\rho(x)}\right)\\
             &=& a_{n-m}^{(m)}\,P_{n+1,m}(x,y) + b_{n-m}^{(m)}\,P_{n,m}(x,y) + c_{n-m}^{(m)}\,P_{n-1,m}(x,y).
\end{eqnarray*}
The result follows from the above relation for $m = 0, 1, 2, \dots, n$, and the vector
 notation \eqref{**}.
\end{proof}

\bigskip

The matrix coefficients of the \emph{second three term relation} for Koornwinder polynomials are tridiagonal 
matrices, as we will prove in the next theorem.

\begin{theorem}\label{yr3t}
The matrix coefficients of the second three term relation \eqref{r3t-i2} for an orthogonal PS
generated by \eqref{kops}
are given by the tridiagonal matrices
\begin{equation}\label{eq:dsubn}
A_{n,2} = \left(\begin{array} {cccc|c}
a^{(0)}_{2,n}   & a^{(0)}_{3,n}   & \ldots      &      0        & 0         \\
a^{(1)}_{1,n-1} & a^{(1)}_{2,n-1} & \ddots      & \vdots        & \vdots     \\
\vdots        & \ddots        & \ddots      & a^{(n-1)}_{3,1} & 0          \\
0             & \ldots        & a^{(n)}_{1,0} & a^{(n)}_{2,0}   & a^{(n)}_{3,0} 
\end{array} \right),
\end{equation}
where
\begin{equation}\label{An,2}
\begin{array}{lllll}
a^{(m)}_{1,n-m} &=& \tilde{c}_m\,\eta_{n-m}^{(m-1)}, &\qquad & 1\le m \le n,\\
a^{(m)}_{3,n-m} &=& \tilde{a}_m\,\delta_{n-m}^{(m)},           &\qquad & 0\le m \le n,\\
\end{array}
\end{equation}

\begin{equation}\label{eq:esubn}
B_{n,2} = \left(\begin{array} {cccc}
b^{(0)}_{2,n}   & b^{(0)}_{3,n}   & \ldots      &      0        \\
b^{(1)}_{1,n-1} & b^{(1)}_{2,n-1} & \ddots      & \vdots        \\
\vdots              & \ddots              & \ddots      & b^{(n-1)}_{3,1} \\
0                   & \ldots              & b^{(n)}_{1,0} & b^{(n)}_{2,0}    
\end{array} \right),
\end{equation}
where
\begin{equation}\label{Bn,2}
\begin{array}{lllll}
b^{(m)}_{1,n-m} &=& \tilde{c}_m\,\theta_{n-m}^{(m-1)}, &\qquad & 1\le m \le n,\\
b^{(m)}_{3,n-m} &=& \tilde{a}_m\,\epsilon_{n-m}^{(m)},           &\qquad & 0\le m \le n-1,\\
\end{array}
\end{equation}

\begin{equation}\label{eq:esucn}
C_{n,2} = \left(\begin{array} {cccc}
c^{(0)}_{2,n}   & c^{(0)}_{3,n}   & \ldots             &      0                \\
c^{(1)}_{1,n-1} & c^{(1)}_{2,n-1} & \ddots             & \vdots                \\
\vdots                & \ddots               & \ddots              & c^{(n-2)}_{3,2} \\
0                     & \ldots               & c^{(n-1)}_{1,1} & c^{(n-1)}_{2,1}   \\ 
\hline
0                     &\ldots                & 0                   & 0 
\end{array} \right),
\end{equation}
where
\begin{equation}\label{Cn,2}
\begin{array}{lllll}
c^{(m)}_{1,n-m} &=& \tilde{c}_m\,\vartheta_{n-m}^{(m-1)}, &\qquad & 1\le m \le n-1,\\
c^{(m)}_{3,n-m} &=& \tilde{a}_m\,\zeta_{n-m}^{(m)},         &\qquad & 0\le m \le n-2,\\
\end{array}
\end{equation}
and

\noindent
\emph{Case I}: if $\rho(x) = r_1\,x+r_0$, where $|r_1|+|r_0|>0$, then 

\begin{equation}\label{e_2}
\begin{array}{lllll}
a^{(m)}_{2,n-m} &=& \tilde{b}_m\,r_1\,a_{n-m}^{(m)},           &\qquad & 0\le m \le n,\\
b^{(m)}_{2,n-m} &=& \tilde{b}_m\,\rho(b_{n-m}^{(m)}) = 
\tilde{b}_m\,(r_1 \,b_{n-m}^{(m)}+ r_0),    &\qquad & 0\le m \le n,\\
c^{(m)}_{2,n-m} &=& \tilde{b}_m\,r_1\,c_{n-m}^{(m)},         &\qquad & 0\le m \le n-1,\\
\end{array}
\end{equation}

\bigskip

\noindent
\emph{Case II}: if $\rho(x)$ is the square root of a polynomial of degree at most 2, then
\begin{equation}\label{caseII}
\begin{array}{lllll}
a^{(m)}_{2,n-m} &=& 0, \quad & 0\le m \le n,\\
b^{(m)}_{2,n-m} &=& 0, \quad & 0\le m \le n,\\
c^{(m)}_{2,n-m} &=& 0, \quad & 0\le m\le n-1.
\end{array}
\end{equation}
\end{theorem}

\bigskip

\begin{proof}
Let us multiply (\ref{kops}) times $y$, then from (\ref{recurrence2}),  we get

\begin{eqnarray}\label{eq:bigrecurrel}
y\,P_{n,m}(x,y)  &=&  \tilde{a}_{m}\,p_{n-m}^{(m)}(x)\,\rho(x)^{m+1}\,q_{m+1}\left(\frac{y}{\rho(x)}\right)
\nonumber\\
& ~ & + \tilde{b}_m\,\rho(x)\,p_{n-m}^{(m)}(x)\,\rho(x)^m\,q_m\left(\frac{y}{\rho(x)}\right) \nonumber \\
& ~ & + \tilde{c}_{m}\,\rho(x)^2\,p_{n-m}^{(m)}(x)\,\rho(x)^{m-1}\,q_{m-1}\left(\frac{y}{\rho(x)}\right).
\end{eqnarray}
Each term of the sum will be considered individually. First, using (\ref{eq:recurrel1}), we obtain
\begin{eqnarray*}
\lefteqn{p_{n-m}^{(m)}(x)\rho(x)^{m+1}q_{m+1}\left(\frac{y}{\rho(x)}\right)} \\
&=&\left[\delta_{n-m}^{(m)}\,p_{n-m}^{(m+1)}(x) + \epsilon_{n-m}^{(m)}\,p_{n-m-1}^{(m+1)}(x) +
\zeta_{n-m}^{(m)}\,p_{n-m-2}^{(m+1)}(x)\right]\\
&~& \times\,\rho(x)^{m+1}\,q_{m+1}\left(\frac{y}{\rho(x)}\right)\\
&=&\delta_{n-m}^{(m)}\,P_{n+1,m+1} + \epsilon_{n-m}^{(m)}\,P_{n,m+1} + 
\zeta_{n-m}^{(m)}\,P_{n-1,m+1},
\end{eqnarray*}
where we omit the variables $(x,y)$ for simplicity.

Now, we consider the second term of the sum in (\ref{eq:bigrecurrel}). If $\rho(x)$ is the square root of 
a polynomial of degree no greater than $2$, then $v_y$ is a symmetric moment functional, and therefore, 
$\tilde{b}_m=0$ for all non--negative integers $m$.

Suppose that $\rho(x)$ is a polynomial of degree $\le 1$, that is, $\rho(x) = 
r_1\,x+r_0$, with $|r_1| + |r_0| >0$. In this case, using (\ref{recurrence1})
\begin{eqnarray*}
\lefteqn{\rho(x)\,p_{n-m}^{(m)}(x)\,\rho(x)^m\,q_m\left(\frac{y}{\rho(x)}\right)} \\
&=& \left\{r_1 \left[a_{n-m}^{(m)}\,p_{n-m+1}^{(m)}(x) + b_{n-m}^{(m)} \, 
p_{n-m}^{(m)}(x) + c_{n-m}^{(m)}\,p_{n-m-1}^{(m)}(x)\right] \right.\\
&~& \qquad + \left.r_0\,
p_{n-m}^{(m)}(x)\right\}\, \rho(x)^m\,q_m\left(\frac{y}{\rho(x)}\right) \\
&=& r_1\,a_{n-m}^{(m)}\,P_{n+1,m} + (r_1\,b_{n-m}^{(m)} + r_0)P_{n,m} + r_1\,c_{n-m}^{(m)}\,
P_{n-1,m}.
\end{eqnarray*}
Next, for $m\ge 1$, including (\ref{eq:recurrel2}) in the third term of the sum in 
\eqref{eq:bigrecurrel}, we deduce
\begin{eqnarray*}
\lefteqn{\rho(x)^2\,p_{n-m}^{(m)}(x)\,\rho(x)^{m-1}\,q_{m-1}\left(\frac{y}{\rho(x)}\right)} \\
&=&\left[\eta_{n-m}^{(m-1)}\, p_{n-m+2}^{(m-1)}(x) +
\theta_{n-m}^{(m-1)} \,p_{n-m+1}^{(m-1)}(x) + \vartheta_{n-m}^{(m-1)}\,p_{n-m}^{(m-1)}(x)\right] \\
&~& \times\, \rho(x)^{m-1}\,q_{m-1}\left(\frac{y}{\rho(x)}\right)\\
&=& \eta_{n-m}^{(m-1)}\,P_{n+1,m-1} + \theta_{n-m}^{(m-1)}\,P_{n,m-1} +
\vartheta_{n-m}^{(m-1)}\,P_{n-1,m-1}.
\end{eqnarray*}
Finally, replace the above expressions into (\ref{eq:bigrecurrel}), and get
\begin{eqnarray*}
y\,P_{n,m}
&=& \tilde{c}_{m}\,\eta_{n-m}^{(m-1)}\,P_{n+1,m-1} + \tilde{b}_{m}\,r_1\,a_{n-m}^{(m)}\,
P_{n+1,m} + \tilde{a}_{m}\,\delta_{n-m}^{(m)}\,P_{n+1,m+1}\\
&~& + \tilde{c}_{m}\,\theta_{n-m}^{(m-1)}\,P_{n,m-1} + \tilde{b}_{m}(r_1\,b_{n-m}^{(m)}+r_0)
P_{n,m} + \tilde{a}_{m}\,\epsilon_{n-m}^{(m)}\,P_{n,m+1}\\
&~& + \tilde{c}_{m}\,\vartheta_{n-m}^{(m-1)}\,P_{n-1,m-1} + \tilde{b}_{m}\,r_1\,c_{n-m}^{(m)}\,
P_{n-1,m} + \tilde{a}_{m}\,\zeta_{n-m}^{(m)}\,P_{n-1,m+1}.
\end{eqnarray*}
\end{proof}

\bigskip

\section{Some examples of bivariate polynomials}\label{5}

In this final section we apply our results to several examples of bivariate orthogonal 
polynomials generated by method described in (\ref{kops}). First, we recover the three term relations for
the four classes of orthogonal polynomials in two variables considered in \cite{Koor75}. These classes
were called by Koornwinder \textit{Class II}, \textit{Class III}, \textit{Class IV} and
\textit{Class V}. The three term relations in \textit{Classes II}, \textit{IV} and \textit{V} were 
explicitly deduced in \cite{X94} in the orthonormal case but using
a different technique.

Moreover, we study the matrix coefficients in the three term relations for other examples 
of bivariate polynomials. First, we study a
new example of bivariate polynomials included in \cite{FPP12}, and finally, we 
deduce three term relations for non positive--definite bivariate classical orthogonal polynomials 
considered in \cite{KLL01} as solutions of a Krall and Sheffer's partial differential equation
(\cite{KS67}) generated by means of this technique.

\bigskip

\subsection{Class II: Orthogonal polynomials on the unit disk}\phantom{fohsf }

In this class, T. Koornwinder constructed the classical orthogonal polynomials in two variables 
on the unit disk, 
$${\bf{B}}^2=\{(x,y)\in\mathbb{R}^2/ x^2 + y^2 \le 1\},$$
the so--called {\it ball polynomials}, associated with the weight function
$$W^{(\mu)}(x,y) = (1-x^2-y^2)^{\mu-1/2}, \qquad \mu > -\frac{1}{2},$$
taking
$$
w_1(x) = w_2(x) = w^{(\mu-1/2,\mu-1/2)}(x) = (1-x^2)^{\mu-1/2},\quad x\in [-1,1],
$$
and $\rho(x) = \sqrt{1-x^2}$, so, we are in Case II. In this case, for $m\ge 0$, we have
$$\rho(x)^{2m+1}w_1(x) = w^{(\mu+m,\mu+m)}(x).$$
Here all of the involved weight functions are symmetric, and, consequently,
$$
b_n^{(m)} = 0, \qquad n, m\ge 0.
$$
Then, orthogonal polynomials on the ball can be defined as
$$P^{(\mu)}_{n,m}(x,y) = P^{(\mu+m,\mu+m)}_{n-m}(x)\,\left(\sqrt{1-x^2}\right)^{m}\,
P^{(\mu-1/2,\mu-1/2)}_{m}
\left(\frac{y}{\sqrt{1-x^2}}\right).$$
In this case, we have $B_{n,1} \equiv B_{n,2} \equiv 0$, and therefore, ball polynomials 
are centrally symmetric.

Following Theorem \ref{xr3t} for the first three term relation and three term recurrence relation 
for classical Jacobi polynomials \eqref{TTRR-J}, the entries of the diagonal 
matrices $A_{n,1}$ and $C_{n,1}$ are respectively given by
$$
\begin{array}{llll}
a_{n-m}^{(m)} &=&  
\dfrac{(n-m+1)\,(n+m+2\mu+1)}{(n+\mu+1)(2n+2\mu+1)},\quad & 0\le m\le n, \\
\\
c_{n-m}^{(m)} &=&  
\dfrac{n+\mu}{2n+2\mu+1},\quad & 1\le m\le n.
\end{array}
$$
On the other hand, using Theorem \ref{yr3t} and the symmetry, we get that the tridiagonal matrices
$A_{n,2}$ and $C_{n,2}$ have zero elements in their diagonals, and again from \eqref{TTRR-J}, 
\eqref{ADJ1-J}, \eqref{ADJ2-J}, we get
$$\begin{array}{lcll}
a_{1,n-m}^{(m)} &=& -\dfrac{(m+\mu-1/2)(n-m+1)_2}{(m+\mu)(n+\mu+1)(2n+2\mu+1)}, \quad &1\le m\le n,\\
\\
a_{3,n-m}^{(m)} &=& \dfrac{(m+1)(m+2\mu)(n+m+2\mu+1)_2}{(2m+2\mu)_2(2n+2\mu+1)(n+\mu+1)},
&0\le m\le n-1,\\
\\
c_{1,n-m}^{(m)} &=& \dfrac{(m+\mu-1/2)(n+\mu)}{(m+\mu)(2n+2\mu+1)}, &1\le m\le n-1,\\
\\
c_{3,n-m}^{(m)} &=& -\dfrac{(m+1)(m+2\mu)(n+\mu)}{(2m+2\mu)_2(2n+2\mu+1)}, &0\le m\le n-1,
\end{array}$$
where
$$(\nu)_{0}=1, \quad (\nu)_{k} = \nu\,(\nu+1)\cdots(\nu+k-1), \quad k\ge 1,
$$ 
denotes the usual Pochhammer symbol.

\bigskip

\subsection{Class III: Orthogonal polynomials on the parabolic biangle}

Class III of bivariate Koornwinder polynomials is obtained by taking Jacobi polynomials on $[0,1]$ 
and Gegenbauer polynomials
\begin{eqnarray*}
w_1(x) &=& (1-x)^\alpha\,x^\beta, \quad x\in [0,1], \quad \alpha, \beta >-1,\\
w_2(y) &=& (1-y^2)^\beta, \quad y\in [-1,1], \quad \beta >-1,\\
\rho(x) &=& \sqrt{x}.
\end{eqnarray*}
Therefore polynomials in two variables defined by
$$
P_{n,m}^{(\alpha, \beta)}(x,y) = P^{(\alpha, \beta+m+1/2)}_{n-m}(2x-1)\,\left(\sqrt{x}\right)^{m}\, 
P_m^{(\beta,\beta)}\left(\frac{y}{\sqrt{x}}\right),
$$
for $\alpha, \beta >-1$, and $0\le m\le n$, are orthogonal with respect to the weight function
$$
W(x,y) = (1-x)^{\alpha}\,(x-y^2)^{\beta},
$$
over the parabolic biangle $\Omega = \{(x,y)\in \mathbb{R}/ y^2 < x < 1\}$.

\medskip

In this case, 
$A_{n,1}$, $B_{n,1}$ and $C_{n,1}$ are full rank diagonal matrix, whose entries are given by
$$\begin{array}{llll}
a_{n-m}^{(m)} &=& \dfrac{(n-m+1)(n+\alpha+\beta+3/2)}{(2n-m+\alpha+\beta+3/2)_2}
,\quad & 0\le m\le n,\\
\\
b_{n-m}^{(m)} &=&  \dfrac{(n+\beta+3/2)(n-m+1)}{2n-m+\alpha+\beta+5/2} - \dfrac{(n+\beta+1/2)(n-m)}{2n-
m+\alpha+\beta+1/2},\quad & 0\le m\le n,\\
\\
c_{n-m}^{(m)} &=&  \dfrac{(n-m+\alpha)(n+\beta+1/2)}{(2n-m+\alpha+\beta+1/2)_2},\quad & 1\le m\le n.
\end{array}$$
For the second three term relation, using \eqref{TTRR-J},\eqref{ADJ1-J-[0,1]}, \eqref{ADJ2-J-[0,1]}, 
 we get
$$a_{1,n-m}^{(m)} = a_{2,n-m}^{(m)} = b_{2,n-m}^{(m)} = c_{2,n-m}^{(m)} = c_{3,n-m}^{(m)} = 0,
$$
and
$$
\begin{array}{llll}
a_{3,n-m}^{(m)} &=& \dfrac{2(m+1)(m+2\beta+1)(n+\alpha+\beta+3/2)}{(2m+2\beta+1)_2(2n-m+\alpha+\beta
+3/2)},\quad & 0\le m\le n,\\
\\
b_{1,n-m}^{(m)} &=& \dfrac{(m+\beta)(n-m+1)}{(2m+2\beta+1)(2n-m+\alpha+\beta+3/2)}, &1\le m\le n,\\
\\
b_{3,n-m}^{(m)} &=& \dfrac{2(m+1)(m+2\beta+1)(n-m+\alpha)}{(2m+2\beta+1)_2(2n-m+\alpha+\beta+3/2)},
&0\le m\le n-1,\\
\\
c_{1,n-m}^{(m)} &=& \dfrac{(m+\beta)(n+\beta+1/2)}{(2m+2\beta+1)(2n-m+\alpha+\beta+3/2)}, &1\le m\le n-1.
\end{array}
$$

\bigskip

\subsection{Class IV: Orthogonal polynomials on the simplex}

Simplex polynomials can also be constructed using Agahanov's technique. In this case, for 
$\alpha, \beta, \gamma > -1$, we consider:
\begin{eqnarray*}
w_1(x) &=&  (1-x)^{\beta+\gamma}\,x^\alpha, \quad x\in [0,1],\\
w_2(y) &=&  (1-y)^\gamma\,y^\beta, \quad y\in [0,1],\\
\rho(x) &=& 1-x,
\end{eqnarray*}
and define the bivariate polynomials
$$
P^{(\alpha,\beta,\gamma)}_{n,m}(x,y) = P^{(
\beta+\gamma+2m+1,\alpha)}_{n-m}(2x-1)\,
(1-x)^{m}\,P^{(\gamma,\beta)}_{m}\left(\frac{2\,y}{1-x}-1\right),
$$
for $0\le m\le n$, which are orthogonal with respect to the weight function
$$W^{(\alpha, \beta,\gamma)}(x,y) = x^{\alpha}\,y^\beta\,(1-x-y)^\gamma,
$$
on the simplex $\mathbf{T}^2=\{(x,y)\in\mathbb{R}^2/ x\ge 0, y \ge 0, 1-x-y\ge 0\}$.

Using the three term recurrence relation for 
Jacobi polynomials on $[0,1]$ \eqref{TTRR-J-[0,1]} and Theorem \ref{xr3t}, we get
$$
\begin{array}{llll}
a_{n-m}^{(m)} &=& \dfrac{(n-m+1)\,(n+m+\alpha+\beta+\gamma +2)}{(2n+\alpha+\beta+\gamma+2)_2},
 &0\le m\le n,\\
\\
b_{n-m}^{(m)} &=& \dfrac{(n-m+\alpha+1)(n-m+1)}{2n+\alpha+\beta+\gamma+3} - 
\dfrac{(n-m+\alpha)(n-m)}{2n+\alpha+\beta+\gamma+1}, &0\le m\le n,\\
\\
c_{n-m}^{(m)} &=& \dfrac{(n+m+\beta+\gamma+1)\,(n-m+\alpha)}{(2n+\alpha+\beta+\gamma+1)_2}, 
&\hspace{-0.5cm}0\le m\le n-1.
\end{array}
$$
Now, following Theorem \ref{yr3t}, the entries of matrix coefficients for the second 
three term relation for simplex polynomials can be deduced from \eqref{TTRR-J-[0,1]},
\eqref{ADJ11-J-[0,1]} and \eqref{ADJ22-J-[0,1]}, and they are given by
\begin{eqnarray*}
a_{1,n-m}^{(m)} &=& \dfrac{(m+\beta)(m+\gamma)(n-m+1)_2}{(2m+\beta+\gamma)_2\,(2n+\alpha+\beta+\gamma+2)_2},
\\
\\
a_{2,n-m}^{(m)} &=& \lambda_m^{(\beta,\gamma)}\,
\dfrac{(n-m+1)(n+m+\alpha+\beta+2)}{(2n+\alpha+\beta+2)_2},
\\
\\
a_{3,n-m}^{(m)} &=& \dfrac{(m+1)(m+\beta+\gamma+1)(n+m+\alpha+\beta+2)_2}{(2m+\beta+\gamma+1)_2
(2n+\alpha+\beta+2)_2},\\
\\
b_{1,n-m}^{(m)} &=& \dfrac{-2(m+\beta)(m+\gamma)(n-m+1)(n+m+\beta+\gamma+1)}{(2m+\beta+\gamma)_2
(2n+\alpha+\beta+\gamma+1)(2n+\alpha+\beta+\gamma+3)}, \\
\\
b_{2,n-m}^{(m)} &=& -\lambda_m^{(\beta,\gamma)}
\left(1-\dfrac{(n-m+\alpha+1)(n-m+1)}{2n+\alpha+\beta+3}+\dfrac{(n-m+\alpha)(n-m)}{2n+\alpha+\beta+1} \right),\\
\\
b_{3,n-m}^{(m)}&=&\dfrac{-2(m+1)(m+\beta+\gamma+1)(n-m+\alpha)(n+m+\alpha+\beta+\gamma+2)
}{(2m+\beta+\gamma+1)_2(2n+\alpha+\beta+\gamma+1)(2n+\alpha+\beta+\gamma+3)},\\
\\
c_{1,n-m}^{(m)}&=&\dfrac{(m+\beta)(m+\gamma)(n+m+\beta+\gamma)_2}{(2m+\beta+\gamma)_2
(2n+\alpha+\beta+\gamma+1)_2},\\
\\
c_{2,n-m}^{(m)}&=&\lambda_m^{(\beta,\gamma)}\dfrac{(n-m+\alpha)(n+m+\beta+\gamma+1)}{
(2n+\alpha+\beta+\gamma+1)_2},\\
\\
c_{3,n-m}^{(m)}&=&\dfrac{(m+1)(m+\beta+\gamma+1)(n-m+\alpha)_2}{
(2m+\beta+\gamma+1)_2(2n+\alpha+\beta+\gamma+1)_2}, 
\end{eqnarray*}
where
$$\lambda_m^{(\beta,\gamma)} = \frac{(m+\beta)m}{2m+\beta+\gamma}-\frac{(m+\beta+1)(m+1)}{2m+\beta+\gamma+2}.$$

\bigskip

\subsection{Class V: Orthogonal polynomials on the square}

Class V of Koornwinder polynomials corresponds with the tensor product of Jacobi polynomials. For 
$0\le m\le n$, the bivariate polynomials
$$
P_{n,m}^{(\alpha, \beta, \gamma, \delta)}(x,y)=P^{(\alpha,\beta)}_{n-m}(x)\,P^{(\gamma,\delta)}_m(y),
$$
are orthogonal with respect to the weight function
$$
W(x,y)=w^{(\alpha,\beta)}(x)\,w^{(\gamma,\delta)}(y),
$$
for $\alpha, \beta, \gamma, \delta >-1$ on the square $[-1,1]\times[-1,1]$, taking $\rho(x)=1$. 

\medskip

In this case, all of the matrix coefficients are diagonal matrices. From \eqref{TTRR-J}, and Theorems 
\ref{xr3t} and \ref{yr3t}, we get for the first 
three term relation
\begin{eqnarray*}
a_{n-m}^{(m)} &=& a_{n-m}^{(\alpha, \beta)}, \quad 0\le m\le n,\\
b_{n-m}^{(m)} &=& b_{n-m}^{(\alpha, \beta)}, \quad 0\le m\le n,\\
c_{n-m}^{(m)} &=& c_{n-m}^{(\alpha, \beta)}, \quad 0\le m\le n-1,
\end{eqnarray*}
and for the second three term relation 
$$a_{1,n-m}^{(m)} = a_{2,n-m}^{(m)} = b_{1,n-m}^{(m)} = b_{3,n-m}^{(m)} = c_{2,n-m}^{(m)}
= c_{3,n-m}^{(m)} =0,
$$
and
\begin{eqnarray*}
a_{3,n-m}^{(m)} &=& a_m^{(\gamma, \delta)}, \quad 0\le m\le n,\\
b_{2,n-m}^{(m)} &=& b_m^{(\gamma, \delta)}, \quad 0\le m\le n,\\
c_{1,n-m}^{(m)} &=& c_m^{(\gamma, \delta)}, \quad 1\le m\le n-1,
\end{eqnarray*}
whose explicit expressions can be seen in \eqref{TTRR-J}.

\bigskip

\subsection{Laguerre--Jacobi two variable orthogonal polynomials}

Using Aga\-ha\-nov's technique, new examples of bivariate orthogonal polynomials were given in 
\cite{FPP12}. Here we study three term relations for \textit{Example 1} in \cite{FPP12}. 
Consider the univariate Laguerre and Jacobi weight functions
\begin{eqnarray*}
w_1(x) &=& x^\alpha \,\mathrm{e}^{-x}, \quad x\in [0,+\infty),\quad \alpha >-1,\\
w_2(y) &=& (1-y)^\beta, \quad y\in [-1,1], \quad \beta>-1,\\
\rho(x) &=& x,
\end{eqnarray*}
and denote by $\{L_n^{(\alpha)}\}_{n\ge0}$ the univariate orthogonal Laguerre polynomials
associated with the weight function $w_1(x)$.

The polynomials
$$P_{n,m}(x,y) = L^{(\alpha+2m+1)}_{n-m}(x)\,x^m\,P^{(\beta,0)}_m\left(\frac{y}{x}\right), \quad 0\le m\le n,$$
are orthogonal with respect to the weight function
$$W(x,y) = x^{\alpha-\beta}\,(x-y)^\beta\,\mathrm{e}^{-x},$$
on the unbounded region $\{(x,y)\in \mathbb{R} / -x < y < x, x>0\}$.

Using Theorem \ref{xr3t}, and \eqref{TTRR-L}, the entries of the matrix coefficients of the first three 
term relation are
$$\begin{array}{llll}
a_{n-m}^{(m)} &=& - (n-m+1),  & 0\le m\le n,\\
b_{n-m}^{(m)} &=& 2n+\alpha+2,  & 0\le m\le n,\\
c_{n-m}^{(m)} &=& - (n+m+\alpha+1), & 0\le m\le n-1.
\end{array}$$
From \eqref{TTRR-J}, \eqref{TTRR-L}, \eqref{ADJ7-L},  \eqref{ADJ8-L}, and using Theorem \ref{yr3t}, 
we can deduce the matrix coefficients for the second three term relation for the new example, and 
they are given by
$$\begin{array}{llll}
a_{1,n-m}^{(m)} &=& \dfrac{2m(m+\beta)(n-m+1)_2}{(2m+\beta)_2}, &1\le m\le n,\\
\\
a_{2,n-m}^{(m)} &=& \dfrac{\beta^2(n-m+1)}{(2m+\beta)(2m+\beta+2)}, &0\le m\le n,\\
\\
a_{3,n-m}^{(m)} &=& \dfrac{2(m+1)(m+\beta+1)}{(2m+\beta+1)_2}, &0\le m\le n,\\
\\
b_{1,n-m}^{(m)} &=& -\dfrac{4m(m+\beta)(n+m+\alpha+1)(n-m+1)}{(2m+\beta)_2},&1\le m\le n,\\
\\
b_{2,n-m}^{(m)} &=& -\dfrac{\beta^2(2n+\alpha+2)}{(2m+\beta)(2m+\beta+2)},&0\le m\le n,\\
\\
b_{3,n-m}^{(m)} &=& -\dfrac{4(m+1)(m+\beta+1)}{(2m+\beta+1)_2},&0\le m\le n-1,\\
\\
c_{1,n-m}^{(m)} &=& \dfrac{2m(m+\beta)(n+m+\alpha)_2}{(2m+\beta)_2},&1\le m\le n-1,\\
\\
c_{2,n-m}^{(m)} &=& \dfrac{\beta^2(n+m+\alpha+1)}{(2m+\beta)(2m+\beta+2)},&0\le m\le n-1,\\
\\
c_{3,n-m}^{(m)} &=&  \dfrac{2(m+1)(m+\beta+1)}{(2m+\beta+1)_2},&0\le m\le n-2.
\end{array}$$

\bigskip

\subsection{Bessel--Laguerre two variable orthogonal polynomials}

Several examples of non--positive definite bivariate orthogonal polynomials generated 
by means of Agahanov's tool can be found in \cite{KLL01} as solutions of 
several Krall and Sheffer's partial differential equations (\cite{KS67}). In particular, Kwon, Lee and
Littlejohn considered the Krall and Sheffer's partial differential equation (5.55)
$$
x^2\,u_{xx} + 2\, x \,y \,u_{xy} + (y^2-y)\,u_{yy} +g(x-1)\,u_x + g(y-\gamma)u_y = n\,(n+g-1)u,
$$
and proved that it has an OPS $\{P_{n,m}(x,y): 0\le m\le n\}_{n\ge0}$ as solutions, 
which cannot be positive--definite, if $g+n\neq 0$ and $g\,\gamma+n\neq 0$, for $n\ge 0$. Here
$$
P_{n,m}(x,y) = B_{n-m}^{(g+2m, -g)}(x)\,x^m\,L_{m}^{(g\gamma-1)}\left(\frac{g\,y}{x}\right),
$$
are given in terms of Bessel and Laguerre polynomials. Observe that we are in \textit{Case I}  
taking $\rho(x) = x/g$, but we can omit the factor $g^{-m}$ 
since the orthogonality is preserved except for a multiplicative factor.

\bigskip

Thus, from Theorem \ref{xr3t} we obtain the explicit expressions for the entries in the 
matrix coefficients
for the first three term relation
$$\begin{array}{llll}
a_{n-m}^{(m)} &=& \dfrac{(n+m+g-1)(-g)}{(2n+g-1)(2n+g)},  & 0\le m\le n,\\
b_{n-m}^{(m)} &=& \dfrac{(2m+g-2)(-g)}{(2n+g-2)(2n+g)},  & 0\le m\le n,\\
c_{n-m}^{(m)} &=& \dfrac{(n-m)\,g}{(2n+g-2)(2n+g-1)}, & 0\le m\le n-1.
\end{array}
$$
For the second three term relation, Theorem \ref{yr3t} provides the entries in the matrix coefficients as
follows,
$$\begin{array}{llll}
a_{1,n-m}^{(m)} &=& - \dfrac{(m+g\gamma-1)\,g^2}{(2n+g-1)_2}, &1\le m\le n,\\
\\
a_{2,n-m}^{(m)} &=& - \dfrac{(2m+g\gamma)\,(n+m+g-1)}{(2n+g-1)_2}, &0\le m\le n,\\
\\
a_{3,n-m}^{(m)} &=& - \dfrac{(m+1)\,(n+m+g-1)_2}{(2n+g-1)_2}, &0\le m\le n,\\
\\
b_{1,n-m}^{(m)} &=& \dfrac{2\,(m+g\gamma-1)\,(g+2m-2)_2}{(2n+g-2)\,(2n+g)}, &1\le m\le n,\\
\\
b_{2,n-m}^{(m)} &=& \dfrac{(2m+g\gamma)\,(g+2m-2)}{(2n+g-2)\,(2n+g)}, &0\le m\le n,\\
\\
b_{3,n-m}^{(m)} &=& - \dfrac{2\,(m+1)\,(n-m)\,(n+m+g-1)}{(2n+g-2)\,(2n+g)}, &0\le m\le n-1,\\
\\
c_{1,n-m}^{(m)} &=& -\dfrac{(m+g\gamma-1)\,(g+2m-2)_2}{(2n+g-2)_2}, &1\le m\le n-1,\\
\\
c_{2,n-m}^{(m)} &=& \dfrac{(2m+g\gamma)\,(n-m)}{(2n+g-2)_2}, &0\le m\le n-1,\\
\\
c_{3,n-m}^{(m)} &=&  - \dfrac{(m+1)\,(n-m-1)_2\,g^2}{(2n+g-2)_2\,(g+2m)_2}, &0\le m\le n-2.
\end{array}
$$

\bigskip

\appendix
\section{Formulas for univariate classical orthogonal polynomials}

In this Appendix, we recall the properties for univariate classical orthogonal polynomials
we have needed in the paper. These properties can be found or can be easily deduced 
from properties in some usual texts like \cite[Chapter 22]{AS72} and \cite[Chapters 4 \& 5]{Sz78}.
For Bessel polynomials we used \cite[Part II]{KF49}. 

\bigskip

\subsection{Classical Jacobi polynomials on $[-1,1]$}

Classical Jacobi polynomials in one variable are orthogonal with respect to the inner product
$$
\langle f,g \rangle_{J} = \int_{-1}^1 f(x)\, g(x)\,w^{(\alpha,\beta)}(x) \, dx,
$$
where the weight function is given by
$$
w^{(\alpha,\beta)}(x) = (1-x)^\alpha\,(1+x)^\beta, \qquad \alpha, \beta >-1,
$$
and the moment functional is given by
$$\langle u^{(\alpha,\beta)},x^n\rangle = \int_{-1}^1 x^n\,w^{(\alpha,\beta)}(x) \, dx.
$$
As usual, we denote by $\{P_n^{(\alpha, \beta)}\}_{n\ge0}$ the orthogonal polynomial sequence associated 
with $w^{(\alpha,\beta)}(x)$, the so--called {\it Jacobi polynomials}, normalized by the condition 
(formula (4.1.1), p. 58, \cite{Sz78})
$$ P_n^{(\alpha,\beta)}(1) = \binom{n+\alpha}{n}.
$$
Next, we collect the formulas that we have used in Section \ref{5}.

\bigskip

\noindent
\textbf{Three term recurrence relation}

\begin{equation}\label{TTRR-J}
x\,P_n^{(\alpha, \beta)}(x) = a_{n}^{(\alpha, \beta)}\,P_{n+1}^{(\alpha, \beta)}(x) +
b_n^{(\alpha, \beta)}\,P_n^{(\alpha, \beta)}(x)+
c_{n}^{(\alpha, \beta)}\,P_{n-1}^{(\alpha, \beta)}(x),
\end{equation}
where
\begin{eqnarray*}
a_n^{(\alpha, \beta)} &=& \frac{2\,(n+1)\,(n+\alpha+\beta+1)}{(2n+\alpha+\beta+1)\,(2n+\alpha+\beta+2)},\\
\\
b_n^{(\alpha, \beta)} &=& \frac{\beta^2-\alpha^2}{(2n+\alpha+\beta)\,(2n+\alpha+\beta+2)},\\
\\
c_n^{(\alpha, \beta)} &=& \frac{2\,(n+\alpha)\,(n+\beta)}{(2n+\alpha+\beta)\,(2n+\alpha+\beta+1)}.
\end{eqnarray*}

\bigskip

\noindent
\textbf{Relation between adjacent families (I)} (\cite[(22.7.18), (22.7.19), p. 782]{AS72})

\begin{equation}\label{ADJ1-J}
P_n^{(\alpha, \beta)}(x) = \delta_{n}^{(\alpha, \beta)}\,P_{n}^{(\alpha+1, \beta+1)}(x) +
\epsilon_{n}^{(\alpha, \beta)}\,P_{n-1}^{(\alpha+1, \beta+1)}(x)+
\zeta_{n}^{(\alpha, \beta)}\,P_{n-2}^{(\alpha+1, \beta+1)}(x),
\end{equation}
where
\begin{eqnarray*}
\delta_{n}^{(\alpha, \beta)} &=& \frac{(n+\alpha+\beta+1)\,(n+\alpha+\beta+2)}{(2n+\alpha+\beta+1)\,
 (2n+\alpha+\beta+2)},\\
\epsilon_{n}^{(\alpha, \beta)} &=& \frac{(\alpha-\beta)\,(n+\alpha+\beta+1)}{(2n+\alpha+\beta)\,(2n+\alpha+\beta+2)
},\\
\zeta_{n}^{(\alpha, \beta)} &=& -\frac{(n+\alpha)\,(n+\beta)}{(2n+\alpha+\beta)
\,(2n+\alpha+\beta+1)}.
\end{eqnarray*}

\bigskip

\noindent
\textbf{Relation between adjacent families (II)} (\cite[(4.5.5) p. 72]{Sz78})
\begin{equation}\label{ADJ2-J}
(1-x^2)\,P_n^{(\alpha+1, \beta+1)}(x) = \eta_{n}^{(\alpha, \beta)}\,P_{n+2}^{(\alpha, \beta)}(x) +
\theta_{n}^{(\alpha, \beta)}\,P_{n+1}^{(\alpha, \beta)}(x)+
\vartheta_{n}^{(\alpha, \beta)}\,P_{n}^{(\alpha, \beta)}(x),
\end{equation}
where
\begin{eqnarray*}
\eta_{n}^{(\alpha, \beta)} &=& -\frac{4\,(n+1)\,(n+2)}{(2n+\alpha+\beta+3)\,
 (2n+\alpha+\beta+4)},\\
\\
\theta_{n}^{(\alpha, \beta)} &=& \frac{4\,(\alpha-\beta)}{(2n+\alpha+\beta+2)\,
(2n+\alpha+\beta+4)},\\
\\
\vartheta_{n}^{(\alpha, \beta)} &=& \frac{4\,(n+\alpha+1)\,(n+\beta+1)}{(2n+\alpha+\beta+2)
\,(2n+\alpha+\beta+3)}.
\end{eqnarray*}

\bigskip

\subsection{Classical Jacobi polynomials on $[0,1]$}

Classical Jacobi polynomials can be defined in the interval $[0,1]$ by means of the
change of variable 
$$\bar{P}_n^{(\alpha, \beta)}(x) = P_n^{(\alpha, \beta)}(2x-1), \quad x\in[0,1].$$ 
In this case, the weight function is given by
$$
\bar{w}^{(\alpha,\beta)}(x) = (1-x)^\alpha\,x^\beta, \quad \alpha, \beta >-1,\quad x\in[0,1].
$$
We translate the needed properties for standard Jacobi polynomials to the interval $[0,1]$.

\bigskip

\noindent
\textbf{Three term recurrence relation}

\begin{equation}\label{TTRR-J-[0,1]}
x\,\bar{P}_n^{(\alpha, \beta)}(x) = \bar{a}_{n}^{(\alpha, \beta)}\,\bar{P}_{n+1}^{(\alpha, \beta)}(x) +
\bar{b}_n^{(\alpha, \beta)}\,\bar{P}_n^{(\alpha, \beta)}(x)+
\bar{c}_{n}^{(\alpha, \beta)}\,\bar{P}_{n-1}^{(\alpha, \beta)}(x),
\end{equation}
where
\begin{eqnarray*}
\bar{a}_n^{(\alpha, \beta)} &=& \frac{a^{(\alpha, \beta)}_n}{2} = 
\frac{(n+1)\,(n+\alpha+\beta+1)}{(2n+\alpha+\beta+1)\,
(2n+\alpha+\beta+2)},\\
\\
\bar{b}_n^{(\alpha, \beta)} &=& \frac{b^{(\alpha, \beta)}_n + 1}{2}=
\frac{(n+\beta+1)(n+1)}{2n+\alpha+\beta+2} - \frac{(n+\beta)\,n}{2n+\alpha+\beta},\\
\\
\bar{c}_n^{(\alpha, \beta)} &=& \frac{c^{(\alpha, \beta)}_n}{2} = \frac{(n+\alpha)\,
(n+\beta)}{(2n+\alpha+\beta)\,
(2n+\alpha+\beta+1)}.
\end{eqnarray*}

\bigskip

\noindent
\textbf{Relation between adjacent families (I)} (\cite[(22.7.19), p. 782]{AS72})
\begin{equation}\label{ADJ1-J-[0,1]}
\bar{P}_n^{(\alpha, \beta)}(x) = \bar{\delta}_{n}^{(\alpha, \beta)}\,P_{n}^{(\alpha, \beta+1)}(x) +
\bar{\epsilon}_{n}^{(\alpha, \beta)}\,\bar{P}_{n-1}^{(\alpha, \beta+1)}(x),
\end{equation}
where
$$
\bar{\delta}_{n}^{(\alpha, \beta)} = \frac{n+\alpha+\beta+1}{2n+\alpha+\beta+1},\qquad 
\bar{\epsilon}_{n}^{(\alpha, \beta)} = \frac{n+\alpha}{2n+\alpha+\beta+1}.
$$

\bigskip

\noindent
\textbf{Relation between adjacent families (II)} (\cite[(22.7.16), p. 782]{AS72})
\begin{equation}\label{ADJ2-J-[0,1]}
x\,\bar{P}_n^{(\alpha, \beta+1)}(x) = 
\bar{\theta}_{n}^{(\alpha, \beta)}\,\bar{P}_{n+1}^{(\alpha, \beta)}(x)+
\bar{\vartheta}_{n}^{(\alpha, \beta)}\,\bar{P}_{n}^{(\alpha, \beta)}(x),
\end{equation}
where
$$
\bar{\theta}_{n}^{(\alpha, \beta)} = \frac{n+1}{2n+\alpha+\beta+2},\qquad
\bar{\vartheta}_{n}^{(\alpha, \beta)} = \frac{n+\beta+1}{2n+\alpha+\beta+2}.
$$

\bigskip

\noindent
\textbf{Relation between adjacent families (III)} (\cite[(22.7.18), p. 782]{AS72})
\begin{equation}\label{ADJ11-J-[0,1]}
\bar{P}_n^{(\alpha, \beta)}(x) = \hat{\delta}_{n}^{(\alpha, \beta)}\,\bar{P}_{n}^{(\alpha+2, \beta)}(x) +
\hat{\epsilon}_{n}^{(\alpha, \beta)}\,\bar{P}_{n-1}^{(\alpha+2, \beta)}(x) + \hat{\zeta}_n^{(\alpha,\beta)}\,
\bar{P}_{n-2}^{(\alpha+2, \beta)}(x),
\end{equation}
where
\begin{eqnarray*}
\hat{\delta}_{n}^{(\alpha, \beta)} &=& \frac{(n+\alpha+\beta+1)(n+\alpha+\beta+2)}{(2n+\alpha+\beta
+1)(2n+\alpha+\beta+2)},\\
\\
\hat{\epsilon}_{n}^{(\alpha, \beta)} &=& -\frac{2(n+\beta)(n+\alpha+\beta+1)}{(2n+\alpha+\beta)
(2n+\alpha+\beta+2)}, \\
\\
\hat{\zeta}_n^{(\alpha,\beta)} &=& \frac{(n+\beta)(n+\beta-1)
}{(2n+\alpha+\beta)(2n+\alpha+\beta+1)}.
\end{eqnarray*}

\bigskip

\noindent
\textbf{Relation between adjacent families (IV)} (\cite[(22.7.15), p. 782]{AS72})
\begin{equation}\label{ADJ22-J-[0,1]}
(1-x)^2\,\bar{P}_n^{(\alpha+2, \beta)}(x) = \hat{\eta}_{n}^{(\alpha, \beta)}\,
\bar{P}_{n+2}^{(\alpha, \beta)}(x)+
\hat{\theta}_{n}^{(\alpha, \beta)}\,\bar{P}_{n+1}^{(\alpha, \beta)}(x)+
\hat{\vartheta}_{n}^{(\alpha, \beta)}\,\bar{P}_{n}^{(\alpha, \beta)}(x),
\end{equation}
where
\begin{eqnarray*}
\hat{\eta}_{n}^{(\alpha, \beta)} &=& \frac{(n+1)(n+2)}{(2n+\alpha+\beta+3)(2n+\alpha+\beta+4)},\\ 
\\
\hat{\theta}_{n}^{(\alpha, \beta)} &=& - \frac{2(n+1)(n+\alpha+2)}{(2n+\alpha+\beta+2)(2n+\alpha+\beta+4)},\\
\\
\hat{\vartheta}_{n}^{(\alpha, \beta)} &=& \frac{(n+\alpha+1)(n+\alpha+2)}{(2n+\alpha+\beta+
2)(2n+\alpha+\beta+3)}.
\end{eqnarray*}

\bigskip

\subsection{Classical Laguerre polynomials}

Univariate classical Laguerre polynomials are orthogonal polynomials associated with the inner product
$$
\langle f,g \rangle_{L} = \int_{0}^{+\infty} f(x)\, g(x)\,w^{(\alpha)}(x) \, dx,
$$
where 
$$
w^{(\alpha)}(x) = x^\alpha\,e^ {-x}, \qquad \alpha >-1.
$$
As usual, we denote by $\{L_n^{(\alpha)}\}_{n\ge0}$ the Laguerre polynomial sequence orthogonal 
with respect to $w^{(\alpha)}(x)$, normalized by the condition (formula (5.1.6), p. 101, \cite{Sz78})
$$
L_n^{(\alpha)}(0)=\binom{n+\alpha}{n}.
$$

\bigskip

\noindent
\textbf{Three term recurrence relation}

\begin{equation}\label{TTRR-L}
x\,L_n^{(\alpha)}(x) = a_{n}^{(\alpha)}\,L_{n+1}^{(\alpha)}(x) +
b_n^{(\alpha)}\,L_n^{(\alpha)}(x)+
c_{n}^{(\alpha)}\,L_{n-1}^{(\alpha)}(x),
\end{equation}
where
\begin{eqnarray*}
a_n^{(\alpha)} = -(n+1),\quad b_n^{(\alpha)} = 2n+\alpha+1,\quad c_n^{(\alpha)} = -(n+\alpha).
\end{eqnarray*}

\bigskip

\noindent
\textbf{Relation between adjacent families (I)} (\cite[(22.7.30), p. 783]{AS72})
\begin{equation}\label{ADJ7-L}
L_n^{(\alpha)}(x) = \delta_{n}^{(\alpha)}\,L_{n}^{(\alpha+2)}(x) +
\epsilon_{n}^{(\alpha)}\,L_{n-1}^{(\alpha+2)}(x)+
\zeta_{n}^{(\alpha)}\,L_{n-2}^{(\alpha+2)}(x),
\end{equation}
where
\begin{eqnarray*}
\delta_{n}^{(\alpha)} = 1,\quad \epsilon_{n}^{(\alpha)} = -2,\quad \zeta_{n}^{(\alpha)} = 1.
\end{eqnarray*}

\bigskip

\noindent
\textbf{Relation between adjacent families (II)} (\cite[(22.7.31), p. 783]{AS72})
\begin{equation}\label{ADJ8-L}
x^2\,L_n^{(\alpha+2)}(x) = \eta_{n}^{(\alpha)}\,L_{n+2}^{(\alpha)}(x) +
\theta_{n}^{(\alpha)}\,L_{n+1}^{(\alpha)}(x)+
\vartheta_{n}^{(\alpha)}\,L_{n}^{(\alpha)}(x),
\end{equation}
where
\begin{eqnarray*}
\eta_{n}^{(\alpha)} &=& (n+1)\,(n+2),\\
 \theta_{n}^{(\alpha)} &=& -2\,(n+\alpha+2)\,(n+1),\\ \vartheta_{n}^{(\alpha)} &=& (n+\alpha+1)\,(n+\alpha+2).
\end{eqnarray*}

\subsection{Bessel polynomials in one variable}

Following \cite{KF49}, univariate classical Bessel polynomials are orthogonal with respect to
the bilinear form
$$
\langle f,g \rangle_{B} = \int_{c}\, f(z)\, g(z)\,w^{(a,b)}(z) \, dz,
$$
where 
$$
w^{(a,b)}(z) =(2\pi i)^{-1}\,z^{a-2}\, e^{-b/z}, \quad a\neq 0, -1,-2, \ldots, \quad b\ne 0,
$$
and $c$ is the unit circle oriented in the counter--clockwise direction. The moment functional 
associated with Bessel polynomials is given by
$$\langle v^{(a,b)},x^n\rangle = \int_c\, z^n\,w^{(a,b)}(z) \, dz.
$$
In this case, $v^{(a,b)}$ is a non--positive definite moment functional.

Let us denote by $\{B_n^{(a,b)}\}_{n\ge0}$ the Bessel polynomial sequence orthogonal 
with respect to $w^{(a,b)}(z)$, normalized by the condition
$$
B_n^{(a,b)}(0)=1.
$$
Next, we recall some useful formulas for orthogonal Bessel polynomials. These expressions can be 
easily deduced from formulas in Part II of \cite{KF49}.

\bigskip

\noindent
\textbf{Explicit formulas for the first two cofficients} (\cite[(34), p. 108]{KF49})
\begin{equation}\label{besselcoef}
B_{n}^{(a,b)}(x) = \sum_{k=0}^n \,\binom{n}{k}\,(n+a-1)_k\, \left(\frac{x}{b}\right)^k = 
k_n^{(a,b)}\,x^n + l_n^{(a,b)}\,x^{n-1} + \cdots
\end{equation}
where 
$$ k_n^{(a,b)} = \frac{(n+a-1)_n}{b^n}, \qquad l_n^{(a,b)} = n\,\frac{(n+a-1)_{n-1}}{b^{n-1}}.$$

\bigskip

\noindent
\textbf{Normalizing factor} (\cite[(58), p. 113]{KF49})
\begin{equation}\label{normbessel}
h_{n}^{(a,b)}=\int_c B_{n}^{(a,b)}(z)^2\,w^{(a,b)}(z)\,dz = 
\frac{(-1)^{n+1}\,n!\,b}{(2n+a-1)\,(a)_{n-1}}.
\end{equation}

\bigskip

\noindent
\textbf{Three term recurrence relation} (\cite[(51), p. 111]{KF49})

\begin{equation}\label{TTRR-B}
x\,B_n^{(a,b)}(x) = a_{n}^{(a,b)}\,B_{n+1}^{(a,b)}(x) +
b_n^{(a,b)}\,B_n^{(a,b)}(x)+
c_{n}^{(a,b)}\,B_{n-1}^{(a,b)}(x),
\end{equation}
where
\begin{eqnarray*}
a_n^{(a,b)} &=& \frac{(n+a-1)\,b}{(2n+a-1)\,(2n+a)},\\
\\
b_n^{(a,b)} &=&-\frac{(a-2)\,b}{(2n+a-2)\,(2n+a)},\\
\\
c_n^{(a,b)} &=& -\frac{n\,b}{(2n+a-2)\,(2n+a-1)}.
\end{eqnarray*}

\bigskip

\noindent
\textbf{Relation between adjacent families (I)} (From \eqref{eq:recurrel1}, \eqref{besselcoef},
 \eqref{normbessel})
\begin{equation}\label{ADJ1-B}
B_n^{(a,b)}(x) = \delta_{n}^{(a,b)}\,B_{n}^{(a+2, b)}(x) +
\epsilon_{n}^{(a,b)}\,B_{n-1}^{(a+2,b)}(x)+
\zeta_{n}^{(a,b)}\,B_{n-2}^{(a+2, b)}(x),
\end{equation}
where
\begin{eqnarray*}
\delta_{n}^{(a,b)} &=& \frac{(n+a-1)\,(n+a)}{(2n+a-1)\,
 (2n+a)},\\
\\
\epsilon_{n}^{(a,b)} &=& \frac{2\,n\,(n+a-1)}{(2n+a-2)\,(2n+a)},\\
\\
\zeta_{n}^{(a,b)} &=& \frac{n\,(n-1)\,b^2}{(2n+a-2)\,(2n+a-1)
\,a\,(a+1)}.
\end{eqnarray*}

\bigskip

\noindent
\textbf{Relation between adjacent families (II)} (From \eqref{eq:recurrel2}, \eqref{besselcoef},
\eqref{normbessel})
\begin{equation}\label{ADJ2-B}
x^2\,B_n^{(a+2,b)}(x) = \eta_{n}^{(a,b)}\,B_{n+2}^{(a,b)}(x) +
\theta_{n}^{(a,b)}\,B_{n+1}^{(a,b)}(x)+
\vartheta_{n}^{(a,b)}\,B_{n}^{(a,b)}(x),
\end{equation}
where
\begin{eqnarray*}
\eta_{n}^{(a,b)} &=& \frac{b^2}{(2n+a+1)\,(2n+a+2)},\\
\\
\theta_{n}^{(a,b)} &=& -\frac{2\,a\,(a+1)\,(n+a)}{(2n+a)\,(2n+a+2)},\\
\\
\vartheta_{n}^{(a,b)} &=& \frac{a\,(a+1)}{(2n+a)
\,(2n+a+1)}.
\end{eqnarray*}


\bigskip


\begin{thebibliography}{10}

\bibitem{AS72} M. Abramowitz, I. A. Stegun, 
        \textit{Handbook of mathematical functions}, 9th printing. 
        Dover, New York, 1972.

\bibitem{Ag65} C. A. Agahanov,
        A method of constructing orthogonal polynomials of two variables
        for a certain class of weight functions (Russian),
        \textit{Vestnik Leningrad Univ.} \textbf{20} (1965), 5--10.

\bibitem{Ch78} T. S. Chihara, 
        \textit{An introduction to orthogonal polynomials}, 
        Mathematics and its Applications \textbf{13},
        Gordon and Breach, New York, 1978.

\bibitem{DX14}
       C. F. Dunkl, Y. Xu,
       \textit{Orthogonal polynomials of several variables}, 2nd edition, 
       Encyclopedia of Mathematics and its Applications, vol. 155, 
       Cambridge Univ. Press, 2014.
       
\bibitem{FPP12} L. Fern\'andez, T. E. P\'erez, M. A. Pi\~nar,
       On Koornwinder classical orthogonal polynomials in two variables,
       \textit{J. Comput. Appl. Math.} 
       \textbf{236} (2012), 3817--3826. 

\bibitem{Koor75} T. H. Koornwinder, 
       \textit{Two--variable analogues of the classical orthogonal polynomials}, 
       in Theory and Application of Special Functions, 
       R. Askey Editor, Academic Press (1975), 435--495.

\bibitem{KF49} H. L. Krall, O. Frink, 
       A new class of orthogonal polynomials: The Bessel polynomials,
       \textit{Trans. Amer. Math. Soc.} \textbf{65} (1949), 100--115.
 
\bibitem{KS67} H. L. Krall, I. M. Sheffer, 
       Orthogonal polynomials in two variables, 
       \textit{Ann. Mat. Pura Appl. (4)} \textbf{76} (1967), 325--376. 
       
\bibitem{KLL01} K. H. Kwon, J. K. Lee, L. L. Littlejohn,
       Orthogonal polynomial eigenfunctions of second--order partial differential equations,
       Trans. Amer. Math. Soc. 353 (2001), 3629--3647.

\bibitem{Su99} P. K. Suetin, 
       \textit{Orthogonal polynomials in two variables}, 
       Gordon and Breach, Amsterdam, 1999.

\bibitem{Sz78} G. Szeg\H{o}, 
       \textit{Orthogonal polynomials}, 4th ed., American
       Mathematical Society Colloquium Publication 23, 
       Providence RI, 1978.
       
\bibitem{X94} Y. Xu,
       A class of bivariate orthogonal polynomials and cubature formula,
       Numer. Math. 69 (1994), 233--241.
             

\end{thebibliography}
\end{document}